\documentclass{birkmult}

\usepackage{amsthm}
\newtheorem{theorem}{Theorem}[section]
\newtheorem{lemma}[theorem]{Lemma}

\usepackage{amssymb}
\usepackage{mathenv}
\usepackage{amsmath}
\usepackage{amsfonts}

\usepackage{verbatim}

\newcommand\inprod[2]{\left<#1,#2\right>}
\renewcommand\Im{\mathrm{Im}}

\newcommand\disc{\mathrm{disc}}

\title{Configurations of Extremal Even Unimodular Lattices}
\author{Scott Duke Kominers}
\address{Department of Mathematics, Harvard University\\c/o 8520 Burning Tree Road, Bethesda, MD, 20817\\
e-mail: \texttt{kominers@fas.harvard.edu}}



\date{November 12, 2007}

\begin{document}

\begin{abstract}We extend the results of Ozeki on the
configurations of extremal even unimodular lattices.  
Specifically, we show that if $L$ is such a lattice 
of rank $56$, $72$, or $96$, then $L$ is generated 
by its minimal-norm vectors.
\end{abstract}
\maketitle
\section{Background}

Even unimodular lattices have been  studied extensively.  There is
a unique  even unimodular lattice of rank $8$, the $E_8$ lattice
(see \cite{Conway:SPLAG}).  Niemeier \cite{Niemeier} gave a complete
characterization of the   even unimodular lattices of rank $24$,
showing in particular that the Leech lattice $\Lambda_{24}$ is the
unique extremal even unimodular lattice of rank $24$.  This and
a similar classification problem were studied by Venkov
\cite{Venkov:32, Venkov:24}.  Venkov~\cite{Venkov:32} and Ozeki
\cite{Ozeki1,Ozeki2,Ozeki3} showed a series of ``configuration''
results for even unimodular lattices.

Specifically, Venkov \cite{Venkov:32} showed that if $L$ is an extremal even unimodular
lattice of rank $32$, then $L$ is generated by its
vectors of minimal norm.  Ozeki \cite{Ozeki1,Ozeki3}
obtained the same result for extremal even unimodular
lattices of rank $32$ and showed the analogous result for extremal
even unimodular lattices of rank $48$.  We extend these results to
cover extremal even unimodular lattices of ranks $56$, $72$, and 
$96$.\footnote{Note that while an extremal even unimodular lattice of rank 
$56$ has been found, the question of whether there exist even unimodular 
lattices of ranks  $72$ and $96$ is still open (see \cite[p. 194]{Conway:SPLAG}).}

\section{Introduction}
A \textit{lattice} of \textit{rank $n$} is a free $\mathbb{Z}$-module
of rank $n$ equipped with a
positive-definite inner product $\inprod{\cdot}{\cdot}:L\times L \to
\mathbb{R}$.   The dual
lattice of $L$, denoted by $L^*$, is given by $$L^*=\{y\in L\otimes
\mathbb{R}\vert
\inprod{y}{x}\in\mathbb{Z}\ \text{for all}\ x\in L\}.$$ 

\subsection{Even Unimodular Lattices}

A lattice $L$ is said to be \textit{integral} if $\inprod{x}{y}\in\mathbb{Z}$
for all $x,y\in L$. An integral lattice $L$ is called \textit{even} if all its
vectors have even norm, that is, $\inprod{x}{x}\in 2\mathbb{Z}$ for all
vectors $x\in L$.   An integral lattice $L$ is said to be \textit{unimodular}
if its dual is itself ($L^*=L$).

Even unimodular lattices exist only in ranks which are multiples
of eight.  Sloane  showed that when $L$ is even
unimodular of rank $n$, the minimal norm of $L$ is bounded by
\begin{equation}\min_{\stackrel{x\in L}{x\neq 0}}\{\inprod{x}{x}\}\leq 2\lfloor
n/24\rfloor+2\label{minnorm}\end{equation}(see \cite[p. 194, Cor.
21]{Conway:SPLAG}).
A rank-$n$ even unimodular lattice is called \emph{extremal} if it attains the
bound (\ref{minnorm}).

\subsection{Theta Functions and Modular Forms}

Each lattice $L$ has an associated \textit{theta function} $\Theta_L$, a
generating function that encodes the norms of the vectors of $L$ in
the form $$\Theta_L(z)=\sum_{x\in L}e^{\pi iz\inprod{x}{x}},$$ where $z$
lies in the complex upper half-plane of complex numbers,
$$\mathcal{H}=\{x\in\mathbb{C}|\Im(x)>0\}.$$

A polynomial $P\in\mathbb{C}[x_1,...,x_n]$ is called
\textit{harmonic} if $\Delta P=0$, where
$\Delta=\sum_{i=1}^n\frac{\partial^2}{\partial x_i^2}$ is the Laplace
operator.  For a lattice $L$ and homogeneous harmonic polynomial $P$
of degree $d$, the \textit{weighted theta series} $\Theta_{L,P}$ is
the function $$\Theta_{L,P}(z)=\sum_{x\in L}P(x)e^{\pi
iz\inprod{x}{x}},$$ where, again, $z\in\mathcal{H}$.

A holomorphic function $f:\mathcal{H}\rightarrow\mathbb{C}$ is
called a \textit{modular form of weight $k$ for $SL_2(\mathbb{Z})$}
if it is holomorphic at $i\infty$ and satisfies the condition
$$f\left(\frac{az+b}{cz+d}\right)=(cz+d)^kf(z)$$ for all
$\left(\begin{array}{cc}a&b\\c&d\end{array}\right)\in
SL_2(\mathbb{Z})/\{\pm1\}$.  If a modular form $f$ vanishes at $z=i\infty$,
it is called a \textit{cusp form}.

A general introduction to modular forms for $SL_2(\mathbb{Z})$
may be found in \cite{Serre:course}.  As shown there, if $M_k$ and
$M_k^0$ are, respectively, the $\mathbb{C}$-vector spaces of modular forms
and cusp forms of weight $k$, then $\dim(M_k)=0$ for $k$ odd, $k<0$, or $k=2$.  Also,
$\dim(M_{2k})=1$ and $\dim(M_{2k}^0)=0$ for $4\leq 2k\leq 10$ and $2k=14$.  Further,
multiplication by the form $\Delta=12^{-3}(E_4^3-E_6^2)$ gives an
isomorphism $$M_{k-12}\stackrel{\sim}{\longrightarrow}M_k^0.$$

When $L$ is even unimodular of rank $n$, the theta function $\Theta_L$
is a modular form of weight $\frac{n}{2}$. As shown in
\cite{Schoeneberg},
the weighted theta function $\Theta_{L,P}$ of an even unimodular
lattice of rank $n$ is a cusp form of weight $\frac{n}{2}+d$,
where $d>0$ is the degree of the polynomial $P$ (see \cite{Ebeling}).

\subsection{Methods}
For an extremal even unimodular lattice $L$ having minimal norm $2m$,
the first $m-1$ coefficients of the weighted theta function
$\Theta_{L,P}$ must vanish, as these coefficients count the vectors of
norms less than $2m$.  We must then have $$\Theta_{L,P}\equiv 0$$ for
several values of $d=\deg P$.  From this, we extract linear conditions
on the vectors in $L$ which allow us to restrict the possible
configurations of $L$.

\section{Dimension 72}\label{72!!}

We adopt Ozeki's notation. For an even unimodular lattice $L$,
we denote by $\Lambda_{2m}(L)$ the set of vectors in $L$ having norm $2m$.
We denote $|\Lambda_{2k}(L)|$ by $a(2k,L)$.  It is clear that the theta series
$\Theta_L$ is given by $\Theta_L(z)=\sum_{k=0}^\infty a(2k,L)q^k$, where $q=e^{2\pi iz}$.
We also use the notation  $$N_i(x)=\textrm{card}(\{y\in
\Lambda_{m(L)}(L)|\inprod{x}{y}=i\}),$$
where $m(L)$ is the minimal norm $\min\{2k>0|a(2k,L)\neq 0\}$.  Using the involution  
$y \longleftrightarrow -y$ of $\Lambda_{m(L)}(L)$, we see that $N_i(x) = N_{-i}(x)$ for 
any $x\in L\otimes\mathbb{R}$.

We have the following configuration result for extremal even
unimodular lattices of rank $72$:
\begin{theorem}\label{72!}
If $L$ is an extremal even unimodular lattice of rank $72$,
then $L$ is generated by $\Lambda_8(L)$.
\end{theorem}

\begin{proof}We partition $L$ into its equivalence
classes modulo $\Lambda_8(L)$; to show the theorem it suffices to show that
any class $[x]\in L/(\Lambda_8(L))$ is represented by a vector $x_0\in [x]$
with $\inprod{x_0}{x_0}\leq 8$.

For any equivalence class $[x_0]$ where $x_0\neq 0$ is a
representative of minimal norm, we have $\inprod{x_0}{x_0}=2t$ for
some $t\geq 5$,
since we have from the fact that $L$ is extremal that
$a(2,L)=a(4,L)=a(6,L)=0$.  In addition, we have the inequality
$$|\inprod{x_0}{x}|\leq 4$$ for all $x\in \Lambda_8(L)$, because if
$\inprod{x_0}{\pm x}>4$ then $L$ contains a vector $x\mp x_0$ of norm
$$\inprod{x\mp x_0}{x\mp x_0}=\inprod{x}{x}\mp
2\inprod{x}{x_0}+\inprod{x_0}{x_0}<\inprod{x_0}{x_0},$$
contradicting the minimality of $x_0$ in $[x_0]$.

Since $|\inprod{x_0}{x}|\leq 4$ for $x\in \Lambda_8(L)$,
we have \begin{equation}
\label{hapki}\sum_{x\in\Lambda_8(L)}\inprod{x}{x_0}^{2k}=2\sum_{i=1}^4
i^{2k}\cdot N_i(x_0),
\end{equation}for all $k>0$.

Since $L$ is an extremal even unimodular lattice whose rank is a
multiple of $24$, we know that $\Lambda_8(L)$ is a spherical
$11$-design \cite[p. 196, Theorem 23]{Conway:SPLAG}.  Hence,
we have that for $1\leq k\leq 5$, the average value of
the degree-$2k$ polynomials $x\mapsto \inprod{x}{x_0}^{2k}$ over
$8^{-\frac{1}{2}}\Lambda_{8}(L)$ equals the average
value of $x\mapsto \inprod{x}{x_0}^{2k}$ over the unit sphere.   

We obtain from the theta series for $L$ that $|\Lambda_{8}(L)|=6218175600$.  We thus have for $1\leq k\leq 5$
\begin{eqnarray}
\sum_{x\in\Lambda_{8}(L)}\nonumber\inprod{x}{x_0}^{2k}&=&
\vert\Lambda_{8}(L)\vert\int_{S^{71}}\inprod{x}{x_0}^{2k}\,d\mu(x)
\\\nonumber&=&6218175600\frac{1\cdot3 \cdots (2k-1)}{72\cdot
(72+2)\cdots (72+2k-2)}8^k
\inprod{x_0}{x_0}^k\\&=&6218175600\frac{1\cdot3 \cdots (2k-1)}{72\cdot
(72+2)\cdots (72+2k-2)}16^kt^k\label{first}.
\end{eqnarray}

Further, the fact that $|\Lambda_{8}(L)|=6218175600$ immediately gives that \begin{equation}
6218175600=|\Lambda_{8}(L)|=N_0(x_0)+2\sum_{i=1}^4N_i(x_0).\label{L8}
\end{equation}

Combining this equation (\ref{L8}) with equations (\ref{hapki}) and 
(\ref{first}) for $1\leq k\leq 5$ yields a
system of six equations in the five unknowns $N_i(x_0)$ ($0\leq i\leq
4$).  We obtain the determinant of the (extended) $6\times 6$ matrix 
for this inhomogeneous system, 
$$2^{25}3^95^47^2 t (168t^4-2800 t^3+17745 t^2-50635 t+54834),$$
which has no real solutions $t> 0$.  In particular, then, we
cannot have any integral solutions for the $N_i(x_0)$ when $t>4$, so
we must instead have $\inprod{x_0}{x_0}=2t=0$, which proves the
result.
\end{proof}

\section{Dimension 96}\label{96!!}

We continue to  use the notations introduced in Section \ref{72!!}.
We also denote by $P_{d,x_0}(x)$ the ``zonal spherical harmonic
polynomial''
of degree $d$, related to the
\textit{Gegenbauer polynomial} by
\begin{equation}\label{GEG}P_{d,x_0}(x)=G_d(\inprod{x}{x_0},
(\inprod{x}{x}\inprod{x_0}{x_0})^{1/2}),\end{equation}
where $G_d(t,1)$ is the Gegenbauer polynomial of  degree $d$ evaluated
at $t$ \cite{Geg}.

We find a configuration result for extremal even unimodular lattices
of rank $96$ analogous to our result for such lattices of rank $72$.
First, however, we will need a lemma which generalizes
part of Corollary 3.4 of \cite{Ebeling}.

\begin{lemma}\label{lemtha}
If $L$ is an extremal even unimodular lattice of rank $n$,
then $\Lambda_{m}(L)$ spans $L\otimes\mathbb{R}$ for any $m>0$ such that 
$\Lambda_{m}(L)$ is nonempty.
\end{lemma}
\begin{proof}
We have $$P_{2,x_0}(x)=\inprod{x}{x_0}^2-\frac{\inprod{x}{x}\inprod{x_0}{x_0}}{n}$$
spherical harmonic in dimension $n$.  For fixed
$x_0\in L\otimes\mathbb{R}$, then,
\begin{equation}\label{USEP2}\sum_{x\in\Lambda_{m}(L)}P_{2,x_0}(x)=0\end{equation}
since $\Lambda_{m}(L)$ is a spherical $3$-design by \cite{Venkov:Res}, as 
$L$ is an extremal even unimodular lattice of rank $n$.

Substituting the explicit form of $P_{2,x_0}$ into (\ref{USEP2})
gives \begin{eqnarray}\label{CoolBean}\sum_{x\in\Lambda_{m}(L)}
\inprod{x}{x_0}^2&=&\frac{1}{n}\left(\sum_{x\in\Lambda_{m}(L)}
\inprod{x}{x}\right)\inprod{x_0}{x_0}\\\nonumber&=&\frac{1}{n}\cdot
m\cdot \vert \Lambda_{m}(L)\vert\inprod{x_0}{x_0}.\end{eqnarray}
If there were some $x_0\in L\otimes\mathbb{R}$ not in the span of
$\Lambda_{m}(L)$, 
then the left-hand side of (\ref{CoolBean})
would vanish, implying $\vert\Lambda_{m}(L)\vert=0$.  By hypothesis, however, 
we have required $\vert\Lambda_{m}(L)\vert> 0$; hence, we have the lemma.
\end{proof}

We now proceed with the configuration theorem for rank-$96$
extremal even unimodular lattices:
\begin{theorem}\label{96!}
If $L$ is an extremal even unimodular lattice of rank $96$,
then $L$ is generated by $\Lambda_{10}(L)$.
\end{theorem}

\begin{proof}We suppose that $\Lambda_{10}(L)$ does not generate $L$,
and consider the set of vectors $L_{10}^*$ dual to the lattice
$L_{10}\subsetneq L$ generated by $\Lambda_{10}(L)$.  We cannot have
$L_{10}^*=L_{10}$, since then $\Lambda_{10}(L)$ would generate an even
unimodular lattice of rank $96$, as its vectors span
$\mathbb{R}^{96}$ by Lemma \ref{lemtha}; in this case
$\Lambda_{10}(L)$ would have to generate all of $L$, since $L$ is
itself even unimodular of rank $96$.

Thus, there is some equivalence class $[x]\in L_{10}^*/L_{10}$ with
minimal-norm representative $x_0$ having rational norm not in $2\mathbb{Z}$.  We
have the inequality $$|\inprod{x_0}{x}|\leq 5$$ for all $x\in
L_{10}$, else we contradict the minimality of $x_0$ in $[x_0]$,
as in the rank-$72$ case.

Since $x_0\in L_{10}^*$, by definition $\inprod{x}{x_0}\in\mathbb{Z}$
for all $x\in \Lambda_{10}(L)$.  Because we also have the constraint
$|\inprod{x_0}{x}|\leq 5$ for $x\in \Lambda_{10}(L)$, we find for all $k>0$
\begin{equation}\label{hapki96}\sum_{x\in\Lambda_{10}(L)}
\inprod{x}{x_0}^{2k}=2\sum_{i=1}^5i^{2k}\cdot N_i(x_0)
\end{equation}where we recall $N_i(x_0)=\{x\in \Lambda_{10}(L)\vert
\inprod{x}{x_0}=i\}$.

It follows from the fact that $\Theta_{L,P_{d,x_0}}(z)$ is a cusp form
with vanishing $q^1$, $q^2$, $q^3$, and $q^4$ coefficients
that $$\Theta_{L,P_{d,x_0}}(z)\equiv0$$ for $d\in\{
2,4,6,8,10,14\}$.  Therefore,
\begin{equation}\label{vanish96}
\sum_{x\in\Lambda_{10}(L)}P_{d,x_0}(x)=0
\end{equation}for each of these $d$.
(Note that for $d\in\{ 2,4,6,8,10\}$ we also have this result directly
from the fact that $\Lambda_{10}(L)$ is a spherical $11$-design, as
these $P_{d,x_0}$ have degree at most $11$.)

From the fact that $|\Lambda_{10}(L)|=565866362880$, we obtain the additional equation
\begin{equation}
565866362880=|\Lambda_{10}(L)|=N_0(x_0)+2\sum_{i=1}^5 N_i(x_0).\label{L10}
\end{equation}

{}Combining this equation (\ref{L10}) with the equations (\ref{vanish96}) for 
$d=\{ 2,4,6,8,10,14\}$, we obtain a system of seven equations in the six
unknowns $N_i(x_0)$ ($0\leq i\leq 5$).  The determinant of this inhomogenous 
system's (extended) matrix factors as
\begin{equation}\label{DET}Ks(s-12)  Q(s),\end{equation} where
$s=\inprod{x_0}{x_0}$,
$K=-2^{42}\cdot3^{15}\cdot5^{10}\cdot7^5\cdot11\cdot13
\cdot17\cdot19\cdot29\cdot47\cdot53\cdot59$, and $Q(s)=25 s^5- 1275
s^4 + 26112 s^3- 267444
s^2+1362720 s-2741760$ is an irreducible quintic polynomial.  The only rational
$s=\inprod{x_0}{x_0}$ which give solutions to our system, therefore, are
the roots $$s=\inprod{x_0}{x_0}\in\{0,12\}$$ of the determinant (\ref{DET}).

But then, any $x\in [x_0]$ is of the form $x_0 + x_1$ for some
$x_1 \in L_{10}$, so we see that
 \begin{equation}\label{xx}\inprod{x}{x}=\inprod{x_0}{x_0}+2\inprod{x_0}{x_1}+
\inprod{x_1}{x_1}.\end{equation}
 Since $\inprod{x_0}{x_0}\in\{0,12\}\subset 2\mathbb{Z}$,
$\inprod{x_0}{x_1}\in\mathbb{Z}$, and $\inprod{x_1}{x_1}\in
2\mathbb{Z}$, we see that $\inprod{x}{x}\in 2\mathbb{Z}$, from which
it follows that $L^*_{10}$ is even.

In particular, then, $L_{10}^*$ is integral, whence it has integral
discriminant.  Because $L_{10}$ has finite index $[L_{10}:L]$ in $L$,
we see that $L_{10}$ must have positive integer discriminant
$$\disc(L_{10})=[L_{10}:L]^2\disc(L)=[L_{10}:L]^2,$$ as well.  But
then, we have both \begin{align*}[L_{10}:L]^2&\in \mathbb{Z}\\
[L_{10}:L]^{-2} & \in \mathbb{Z},
\end{align*} from which it follows that $[L_{10}:L]=1$, so that $L_{10}=L$.
\end{proof}

\section{Dimension 56}

We use a combination of the techniques developed in Sections \ref{72!!}
and \ref{96!!} to find a configuration result for extremal even unimodular
lattices of rank $56$.

\begin{theorem}\label{56!}
If $L$ is an extremal even unimodular lattice of rank $56$,
then $L$ is generated by $\Lambda_{6}(L)$.
\end{theorem}
\begin{proof}
We partition $L$ into equivalence classes modulo $\Lambda_6(L)$.
As in the rank-$72$ case, it suffices to show that any class
$[x]\in L/(\Lambda_6(L))$ is represented by a vector $x_0\in [x]$
with $\inprod{x_0}{x_0}\leq 6$.

Let $x_0$ a minimal representative of any equivalence class $[x_0]\in
L/(\Lambda_6(L))$.
We have $\inprod{x_0}{x_0}=2t$ for some  $t\geq 4$ if $x_0\neq 0$
and we have the inequality
$$|\inprod{x_0}{x}|\leq 3$$ for all $x\in \Lambda_6(L)$.

{}From this condition, we obtain the equation \begin{align}
\label{hapki56}\sum_{x\in\Lambda_6(L)}\inprod{x}{x_0}^{2k}=2\sum_{i=1}^3\cdot
i^{2k}\cdot N_i(x_0),
\end{align}for all $k>0$.

It follows from the fact that $\Theta_{L,P_{d,x_0}}(z)$ is a cusp form
with vanishing $q^1$ and $q^2$ coefficients
that $$\Theta_{L,P_{d,x_0}}(z)\equiv0$$ for $d\in\{ 2,4,6,10\}$.  This
gives
\begin{equation}\label{vanish56}
\sum_{x\in\Lambda_{6}(L)}P_{d,x_0}(x)=0,
\end{equation}for each $d\in\{ 2,4,6,10\}$.

Combining the fact that $$N_0(x_0)+2\sum_{i=1}^{3}N_i(x_0)=|\Lambda_6(L)|=15590400$$ with (\ref{hapki56}) and (\ref{vanish56}) for
$d\in\{2,4,6,10\}$, we obtain a system of five equations in the four
unknowns  $N_i(x_0)$ ($0\leq i\leq 3$).  This system's (extended) matrix has determinant
\begin{equation}
\label{DET56} -Kt^2Q(t) \end{equation} where
$K=2^{28}\cdot3^8\cdot 5^3\cdot 7\cdot 29\cdot 31$ and $Q(t)$ is the
irreducible sextic {\small$$ Q(t)=8976 t^6-76120 t^5+104624 t^4+533337
t^3 - 972400 t^2- 1952280 t+3644256.$$}Thus, the only integral
solution to the system is $t=0$.  But then we cannot have
$\inprod{x_0}{x_0}=2t>6$, so we are done.\end{proof}

\section*{Acknowledgements.}This work was partially
supported by a fellowship from the 2006 Harvard College Program for
Research in Science and Engineering.

The author is extremely grateful to Professor Noam 
D. Elkies, under whose direction this research was conducted, for his instruction, 
commentary on the work, and remarks on several earlier drafts of this article.  He
also thanks Gabriele Nebe and an anonymous referee for their helpful comments and 
suggestions on earlier drafts of this article.


\begin{thebibliography}{99}
\bibitem{Geg}C. Bachoc, G. Nebe, B. Venkov, Odd unimodular lattices of
minimum 4, \textit{Acta Aritmetica} \textbf{101}(2) (2002).

\bibitem{Conway:SPLAG}J. H. Conway, N. J. A. Sloane, \textit{Sphere
Packing, Lattices and Groups}, 3rd edition (Springer-Verlag,
1999).

\bibitem{Niemeier}H.-V. Niemeier, Definite quadratische Formen der
Dimension 24 und Diskriminante 1 (German), \textit{J. Number Thy.}
\textbf{5} (1973) 142-178.

\bibitem{Ebeling}W. Ebeling,  \textit{Lattices and Codes}, 2nd
edition (Vieweg, 2002).

\bibitem{Ozeki1}M. Ozeki, On even unimodular positive definite
quadratic lattices of rank $32$, \textit{Math.~Z.} \textbf{191}(2)
(1986) 283-291.

\bibitem{Ozeki2}M. Ozeki, On the structure of even unimodular extremal
lattices of rank $40$, \textit{Rocky Mtn. J. Math.} \textbf{19}(3) (1989)
847-862.

\bibitem{Ozeki3}M. Ozeki, On the configurations of even unimodular
lattices of rank $48$, \textit{Arch. Math.} \textbf{46}(1) (1986) 247-287.

\bibitem{Schoeneberg}A. Schoeneberg, Das verhalten von mehrfachen 
thetareihen bei modulsubstitutionen (German),
\textit{Math. Ann.} \textbf{116}(1) (1939) 511-523.

\bibitem{Serre:course}J.-P. Serre, \textit{A Course in Arithmetic}
(Springer-Verlag, 1973).

\bibitem{Venkov:32}B. B. Venkov, Even unimodular Euclidean lattices of
dimension 32 (Russian), \textit{Zap. Nauchn. Sem. Leningrad. Otdel. Math.
Inst. Steklov (LOMI)} \textbf{116} (1982) 44-55.

\bibitem{Venkov:24}B. B. Venkov, On the classification of integral
even unimodular 24-dimensional quadratic forms, \textit{Proc. Steklov Inst.
Math.} \textbf{148} (1980) 63-74.

\bibitem{Venkov:Res}B. B. Venkov,  R\'{e}seaux et designs sph\'{e}riques
(French), in \textit{R\'{e}seaux Euclidiens, Designs Sph\'{e}riques et 
Formes Modulaires}, Monogr. Enseign. Math., Vol. 37 (Enseignement Math.,
 2001), pp. 10-86. 

\end{thebibliography}
\end{document}